\documentclass[12pt]{article}

\usepackage{diagbox}
\usepackage{mathtools}
\usepackage{bbm}
\usepackage{latexsym}
\usepackage{epsfig}
\usepackage{amsmath,amsthm,amssymb,enumerate,hyperref}
\usepackage[active]{srcltx}
\usepackage{color}
\def\red{\color{black}}

\newcounter{rot}

\def\a{\alpha} \def\b{\beta}  
\def\e{\varepsilon}    \def\g{\gamma}
\def\G{\Gamma}  
 \def\th{\theta}    
\def\La{\Lambda}   \def\p{\pi}

\newtheorem{theorem}{Theorem}
\newtheorem{lemma}[theorem]{Lemma}

\newcommand{\rdup}[1]{{\left\lceil #1\right\rceil }}

\newcommand{\brac}[1]{\left(#1\right)}

\newcommand{\bfrac}[2]{\left(\frac{#1}{#2}\right)}

\newcommand{\set}[1]{\left\{#1\right\}}

\def\E{\mathbb{E}}
\def\Var{\mathbb{V}\text{ar}}
\def\Pr{\mathbb{P}}

\newcommand{\ignore}[1]{}

\newcommand{\beq}[2]{\begin{equation}\label{#1}#2\end{equation}}

\usepackage{tikz}
\usetikzlibrary{decorations.pathmorphing}
\usetikzlibrary{positioning}
\usetikzlibrary{arrows,automata}
\usetikzlibrary{shapes.misc}
\usetikzlibrary{backgrounds}
\usetikzlibrary{arrows,shapes}

\def\vK{\vec{K}}

\def\dg{\rm{dang}}
\def\bdg{\overline{\dg}}
\def\M{\text{Maker}}
\def\B{\text{Breaker}}
\begin{document}
\author{Alan Frieze\thanks{Research supported in part by NSF grant DMS1661063. Corresponding author.} and Wesley Pegden\thanks{Research supported in part by NSF grant DMS1363136}\\Department of Mathematical Sciences\\Carnegie Mellon University\\Pittsburgh PA 15213}

\title{Maker Breaker on Digraphs}
\maketitle

\begin{abstract}
We study two biased Maker-Breaker games played on the complete digraph $\vec{K}_n$. In the strong connectivity game, Maker wants to build a strongly connected subgraph. We determine the asymptotic optimal bias for this game viz. $\frac{n}{\log n}$.  In the Hamiltonian game, Maker wants to build a Hamiltonian subgraph.  We determine the asymptotic optimal bias for this game up to a constant factor.
\end{abstract}
{\bf Keywords:} Maker-Breaker, Digraphs, Strong connectivity, Hamiltonicity.
\section{Introduction}
We consider some biased Maker-Breaker games played on the complete digraph $\vK_n$ on $n$ vertices. This is in contrast to the large literature already existing on games played on the complete graph $K_n$. For a very nice summary of the main results in this area, we refer the reader to the monograph by Hefetz, Krivelevich, Stojakovi\'c and Szabo \cite{Book}. A typical example of such a game is the biased connectivity game first studied by Chv\'atal and Erd\H{o}s \cite{CE}. It is palyed on the complete graph $K_n$ and player Maker chooses an edge and then player Breaker chooses $b$ edges. Maker and Breaker's choices being disjoint. Maker will try to ensure that at the end of play, the graph $G_M$ induced by her edges is connected and Breaker will try to prevent this. It is clear that there is some threshold $b_0$ say, such that if $b<b_0$ then Maker will succeed and if $b\geq b_0$ then Breaker will win. Estimating $b_0$ is the challenge in this area. Chv\'atal and Erd\H{o}s gave an upper bound of $b_0\leq (1+\e)\frac{\log n}{n}$ in the connectivity game and Gebauer and Szab\'o \cite{GS} proved a lower bound of $b_0\geq (1-\e)\frac{\log n}{n}$. We can rephrase the connectivity game as that Maker wants to build a spanning tree of $K_n$ and Breaker wants to prevent this. As such we see that there is a corresponding game for any possible subgraph of $K_n$. In particular, Maker might wish to build a Hamilton cycle. This Hamiltonicity game was solved by Krivelevich \cite{K}.

Our aim here is to analyse the directed versions of the Connectivity Game and the Hamiltonicity Game. 

In the games analysed below, Maker goes first, claiming an edge of $\vK_n$. Breaker then claims $b$ edges and so on, with Maker and Breaker taking one and $b$ edges respectively until there are no edges left to take. In addition, Maker and Breaker must claim disjoint sets of edges. Maker is aiming to construct a digraph with certain properties and Breaker is aiming to prevent this. The properties involved are monotone increasing and so there is a critical bias, $b_0$ say, such that if $b<b_0$ then Maker will win and if $b\geq b_0$ then Breaker will win. We will consider two properties here: strong connectivity and Hamiltonicity. We let $D_M,D_B$ denote the digraphs with vertex set $[n]$ and the edges taken by Maker, Breaker respectively. Maker wins the strong connectivity game if on termination $D_M$ is strongly connected.  Maker wins the Hamiltonian game if on termination $D_M$ is Hamiltonian. 
\begin{theorem}\label{th1}
Let $\e>0$ be arbitrarily small and $n\geq n_\e$ sufficiently large. Then Breaker wins the strong connectivity game if $b\geq \frac{(1+\e)n}{\log n}$ and Maker wins if $b\leq \frac{(1-\e)n}{\log n}$.
\end{theorem}
\begin{theorem}\label{th2}
Let $\e>0$ be arbitrarily small and $n\geq n_\e$ sufficiently large. Then Breaker wins the Hamiltonian game if $b\geq \frac{(1+\e)n}{\log n}$. Furthermore, there exists an absolute constant $\b>0$ such that Maker wins if $b\leq \frac{\b n}{\log n}$.
\end{theorem}
The structure of the paper is as follows: Section \ref{degree} deals with a game played on an {\em undirected} $n$-regular graph $G$. We show how Maker can exert some control over the minimum degree in the graph $G_M$ induced by Maker's edges and over the maximum degree in the graph $G_B$ induced by Breaker's edges. We apply the theorem in the context of the bipartite graph $K_{n,n}$ using the natural relationship between a digraph $D=(V,E)$ and the bipartite graph $\G=(V,V,A)$ where $\set{i,j}\in A$ if and only $(i,j)\in E$. 

Section \ref{BW1} verifies the Breaker win in the Strong Connectivity and Hamilton cycle games for sufficiently large $b$. Section \ref{MW1} verifies the Maker win in the Strong Connectivity game for $b$ not too large.

Section \ref{Ham} verifies the Maker win in the Hamilton Cycle game for $b$ not too large.
\section{Degree bound}\label{degree}
{\bf Notation}
We let $d_M^+(v),d_M^-(v)$ denote the out-degree, in-degree of vertex $v$ in $D_M$ for $v\in [n]$. We define $d_B^+,d_B^-$ similarly. We let $E_M,E_B$ denote the edges claimed by Maker and Breaker respectively, on termination.

The following Theorem is a straightforward generalisation of results from Chapter 5 of Hefetz, Krivelevich, Stojakovi\'c and Szabo \cite{Book}. Let $b=\frac{\b n}{\log n}$ be \B's bias, where $\b<1$ is a constant. Note that it deals with undirected graphs. As already mentioned above, we apply it to $K_{n,n}$ and we will see that it gives control for Maker over in- and out-degrees.
\begin{theorem}\label{th1a}
Let $G=(V,E)$ be an $n$-regular graph and let $\a\in (\b,1)$ and suppose that  $2/\a\leq K\leq \th \log n$ where $2\th<\frac{\a-\b}{\b}$ is a constant. Then the following holds: Maker has a randomised strategy that with positive probability can in at most $K|V|$ rounds ensure that Maker's graph has minimum degree $K$ and Breaker's graph has maximum degree at most $\a n$. Furthermore, Maker always chooses a vertex $v$ and then randomly chooses an edge incident with $v$ from a set of size at least $(1-\a)n$.
\end{theorem}
The proof of this involves a minor modification of the proof in \cite{Book}. We have for completeness provided a condensed proof in an appendix. Of course, having a randomized strategy in this context, also means implies the existence of a deterministic strategy.

Now a digraph $D$ on vertex set $[n]$ can be associated with a bipartite graph $G$ on vertex set $A\cup B$ where $A=\set{a_1,\ldots,a_n},B=\set{b_1,\ldots,b_n}$ and where oriented edge $(i,j)$ is replaced by the edge $\set{a_i,b_j}$. In this way the out-degree of $k$ in $D$ is the degree of $a_k$ in $G$ and the in-degree of $k$ is the degree of $b_k$ in $G$. It follows from Theorem \ref{th1a} that \M\ can ensure that $D_M$ has minimum in- and out-degree at least $K$ after at most $2Kn$ rounds. And that $D_B$ has maximum in- and out-degree at most $\a n$.
\section{Strong Connectivity}\label{BW1}
\subsection{Breaker win}
We now consider the game to be played on the complete bipartite graph $K_{n,n}$ where the bipartition is $A\cup B$ with $|A|=|B|=n$. Breaker's aim is to claim all the edges incident with some vertex $a\in A$. This is essentially the {\em box game} of Chv\'atal and Erd\H{o}s \cite{CE}. We let {\em box} $A_i=\set{\set{i,b}:b\in B}$ for $i\in A$. \B\ claims $b$ elements from the boxes and \M\ claims one whole box in each turn. The claimed upper bound follows from Theorem 2.1 of \cite{CE}. 

Note that this also verfies the \B\ win in Theorem \ref{th2}.
\subsection{Maker win}\label{MW1}
Because \M\ chooses neighbors randomly, small sets must have edges entering and leaving.
\begin{lemma}\label{exp}
Suppose that $K\geq (2-2\log(1-\a))/\a$ is sufficiently large with respect to $\a$. Then, w.h.p., $S\subseteq [n]$, $|S|\leq (1-\a)^2n$ implies that 
\[
\set{(i,j)\in E_M:i\in S,j\notin S}\neq \emptyset\text{ and }\set{(i,j)\in E_M:i\notin S,j\in S}\neq \emptyset.
\]
\end{lemma}
\begin{proof}
The probability that there exists a set violating the condition in the lemma is at most
\[
2\sum_{s=K}^{(1-\a)^2n}\binom{n}{s}\bfrac{s}{(1-\a)n}^{Ks}\leq 2\sum_{s=K}^{(1-\a)^2n}\brac{\frac{ne}{s}\bfrac{s}{(1-\a)n}^K}^s=o(1).
\]
\end{proof}
Assume now that $\b=1-\e$ is a close to one and that $\b=(1+\a)/2$. Now consider the Directed Acyclic Graph (DAG)  with one vertex for each strong component of $D_M$ in which there is an edge $(A,B)$ if there is an edge in $D_M$ directed from $A$ to $B$. (A DAG is a digraph without directed circuits.) We observe that w.h.p. each source and sink in $D_M$ must be associated with a subset of $[n]$ of size at least $(1-\a)^2n$. This follows directly from Lemma \ref{exp}. A smaller sink would have an edge oriented from it to another strong component, contradiction.

It follows that w.h.p. after $2Kn$ rounds, \M\ can make $D_M$ strongly connected in a further $\rdup{(1-\a)^{-4}}$ rounds by adding an edge from each sink to each source. There will be by construction $\Omega(n^2)$ choices of edge available for each such pair and \B\ can only claim $o(n)$ edges in this number of rounds. This completes the proof of Theorem \ref{th1}
\section{Hamiltonicity}\label{Ham}
We show that w.h.p. the digraph constructed  by \M\ is Hamiltonian. For each $v\in [n]$ there are sets in-neighbors $IN(v)$ and out-neighbors $OUT(v)$ of size $K=\th\log n$, where each of the $2n$ $K$-sets have been chosen uniformly from sets $A(v),B(v)$ of size $(1-\a)n$. The sets $A(v),B(v),v\in [n]$ are chosen adversarially.

Our analysis assumes that $\a$ is sufficiently small and $\th$ is sufficiently large, given $\a$, $\th>10/\a^2$ will suffice. Note that we require $2\th\b\leq\a$ so that $2K\b n/\log n\leq\a n$ in order not to violate \M's choices. We also need $\th<\frac{\a-\b}{\b}$, which is required by Theorem \ref{th1a}. This makes $\b$ small compared to $\a$.

We will follow an approach similar to that of Angluin and Valiant \cite{AV}.  We choose an arbitrary vertex denoted $s_P$ to start and at any point during the execution of the algorithm we have (i) a path $P$ that begins at $s_P$ and ends at some vertex $f_P$, (ii) a cycle $C$ disjoint from $P$ and (iii) a set $U=[n]\setminus (V(P)\cup V(C))$. We let $P[a,b]$ denote the sub-path of $P$ that goes from $a$ to $b$. At certain points one of $P$ or $C$ may be empty and we denote this by $\La$.  

The execution of the algorithm does not require the whole of $OUT(v)$ be known at the start. As the algorithm progresses it learns more and more about the contents of the lists $OUT(v), v\in [n]$. We will assume that $OUT(v)$ is kept as a randomly ordered list and that its contents are accessed through a pointer $out(v)$. Initially only one vertex of $OUT(v)$ is known for each $v\in [n]$ i.e. the first vertex of the list and this is pointed to by $out(v)$. These pointers are updated to the next vertex in the list, after a selection is made from $OUT(v)$. The actual choices of vertices in the OUT sets are exposed one by one as the algorithm progresses. This is usually referred to as {\em deferred decisions}. Thus the algorithm will ask for $out(v)$. A random choice will be made and then move the pointer $out(v)$ to the next place in the list. Imagine then that the $OUT(v)$ are ordered lists of boxes, each containing a random integer (from some large set). 

Maker's strategy can be summarised as follows: the aim is execute the algorithm until $C$ is a Hamilton cycle. While $|U|\geq 2\a n$, she just tries to grow $P$ by adding the edge $(f_P,out(f_P)$ if $out(f_P)\in U$ to the tail end of $P$. Once $|U|<2\a n$ there are more possibilities. If $C=\La$ then she examines $out(f_P)$. If $y=out(f_P)\in U$ then $P$ grows. If $y\in P$ is far from $f_P$ along $P$ then the edge of $P$ pointing into $y$ is deleted and we are left with a smaller $P$ and a cycle $C$ of size at least $2\a n$. If $C\neq\La$ then she examines $y=out(f_P)$. If $y\in U$ then $P$ grows and $C$ is unchanged. If $y\in C$ then she creates a new path from $P\cup C$ by deleting the edge of $C$ pointing into $y$. This continues until $P\cup C=[n]$ and then w.h.p. after $O(n)$ more steps we find that $|C|=n$. 

Next let
\[
\bar{U}^*=\set{v\notin U:\;\exists u\in U\ s.t.\ v\in IN(u)}.
\]
A general step of the process proceeds as follows: we begin with $P=(s_P=f_P),C=\La$ and $U=[n]\setminus \set{s_P}$ for an arbitrary choice of vertex $s_P$. \\
For $x\in P\cup C$ we let $\p(x)$ denote the unique vertex $z$ such that $(z,x)$ is an edge of $P\cup C$. 

While $|U|\geq 2\a n$, we simply try to grow $P$ by attaching an edge $(f_P,u)$ where $u=out(f_P)\in U$. 

If $|U|<2\a n$ and $C=\La$ then we try to create a large cycle $C$ by adding an edge from $f_P$ to $y\neq s_P, y\in P$. We then delete the previous edge $(x,y)$ of $P$ that points to $y$. The size of $U$ is decreased when $C=\La$ and the vertex $x\in \bar{U}^*$. In this case we can extend $P$ by adding an edge $(x,u),u\in U$ where $x\in IN(u)$. If $C\neq \La$ then we wait until we add another edge $(f_P,z)$ where $z\in C$ and then make one long path, perhaps reducing $|U|$ by adding an edge pointing into $U$.

The reader will notice that we avoid adding the edge $(f_P,s_P)$ if $s_P=out(f_P)$. This is just a matter of convenience. It saves adding a case. The details are as follows: Recall that in what follows $P$ goes from its start $s_P$ to its finish $f_P$ and $y=out(f_P)$.
\begin{enumerate}[{\bf C{a}se 1}]
\item $|U|\geq 2\a n$ (and so $C=\La$). If $y\in U$ then $P\gets P+(f_P,y)$ and $U\gets U\setminus \set{y}$.
\item $|U|<2\a n$.
\begin{enumerate}
\item If $C\neq\La$ and $y\in C$ then $P\gets P[s_P,y]+C[y,\p(y)]$, $C\gets \La$.
\item If $C=\La$ and $y\in P$ and $x=\p(y)\notin \bar{U}^*$  and $y\neq s_P$ is distance at least $2\a n$ from $f_P$ along $P$ then $P\gets P[s_P,x]$ and $C\gets P[y,f_P]+(f_P,y)$.
\item If $C=\La$ and $y\in P$ and $x=\p(y)\in \bar{U}^*$  and $y\neq s_P$ is distance at least $2\a n$ from $f_P$ along $P$ then $P\gets P[s_P,x]+(x,u)$ and $C\gets P[y,f_P]+(f_P,y)$ where $u\in U$ and $x\in IN(u)$. $U\gets U\setminus \set{u}$.\\
\end{enumerate}
\end{enumerate}

\vspace{-.35in}
If none of these cases are applicable, then move $out(f_P)$ to the next vertex on its list. 

\bigskip
It follows that $|C|=0$ or $|C|\geq 2\a n$ throughout. The pointers $out$ are updated if necessary to the next vertex on the list, if they are used in a step. Also, the above procedure fails if it reaches the end of a vertex list before creating a Hamilton cycle. 

Next let $X_i$ be the number of edges examined in order to increase $|P|+|C|$ from $i$ to $i+1$. Note that all random choices can be ascribed to a choice of $out(f_P)$. We now discuss the distribution of the $X_i$.
\begin{enumerate}[(a)]
\item If $|U|\geq 2\a n$ then $X_i$ is dominated by the geometric random variable $Geo(p_1)$ where $p_1=\frac{|U|-\a n}{n}\geq \a$.\\ \\
This is because $f_P$ has at least $|U|-\a n$ choices available to it in $U$ for the next choice of vertex in $OUT(f_P)$. A step here means opening an OUT box and looking inside. 
\item 
Now consider the case where $C=\La$ and $1\leq |U|< 2\a n$. At this point we need a lower bound on the size of $|\bar{U}^*|$. In Section \ref{4.1} below, we will prove the following bound that holds w.h.p. throughout the process:
\beq{CPstar}{
|\bar{U}^*|\geq \begin{cases}n/20&|U|\geq \frac{n}{\th\log n}.\\\th^{1/2}|U|\log n&|U|<\frac{n}{\th\log n}.\end{cases}
}
\end{enumerate}
We will now use the above to estimate how long it takes to finish the process. We will justify it at the end of this section. Let us first ignore the sizes of the sets $OUT(v),v\in [n]$ and deal with this issue later. Each $X_i$ can be coupled with and bounded by an independent geometric random variable with probability of success, $p_i$ say. We will see that w.h.p. we obtain $P\cup C=[n]$ in less than $O(n)$ trials. Here a trial means exposure of $out(v)$. Our high probability bound on the number of trials will follow from the Chebyshev inequality and from the fact that $\E(Geo(p))=\frac{1}{p}$ and $\Var(Geo(p))=\frac{1-p}{p^2}$. 

When $|U|\geq 2\a n$ we take $p_i\geq \a$. When $\frac{n}{\th\log n}\leq |U|\leq 2\a n$ we can take $p_i=\a(1-\a)(1/40-\a)$. When $|U|<\frac{n}{\th\log n}$ we take $p_i=\a(1-\a)\th^{1/2}|U|/6n$. We can argue this as follows. Suppose that $C=\La$. If at least half of $\bar{U}^*$ is on $P$ and is further than $2\a n$ from $f_P$ then we are in good shape. In this case there is a probability $p$ at least $(1-\a)/2n$ times the RHS of \eqref{CPstar} of  reducing $|U|$ by one. Otherwise, with probability at least $(1-\a)/2$, we create a cycle of size at least $n/2$ and then with probablity at least 1/3 we will be in good shape after the next choice of $out(f_P)$. So, when $|U|<2\a n$ we reduce $|U|$ by one in at most four steps, with probability at least $p/6$. If $C\neq \La$ then there is a probability of at least $\a$ of making it equal to $\La$ in one step. Thus if $T$ denotes the total number of trials then we have, 
\begin{align*}
\E(T)&\leq \frac{n}{\a}+4\sum_{u=n/\th\log n}^{2\a n}\frac{6}{\a(1-\a)(1/40-\a)}+ 4\sum_{u=1}^{n/\th\log n}\frac{n}{6\a(1-\a)u\th^{1/2}\log n}=\Theta(n).\\
\Var(T)&\leq \frac{n}{\a}+16\sum_{u=n/\th\log n}^{2\a n}\frac{36}{\a^2(1-\a)^2(1/40-\a)^2}+ 16\sum_{u=1}^{n/\th\log n}\frac{n^2}{36\a^2(1-\a)^2u^2\th\log^2n}=o(n^2).
\end{align*}
The Chebyshev inequality now implies that w.h.p. the number of trials needed is at most $Cn$ for some constant $C=C(\a)>0$.

We now deal with the sizes of the sets $OUT(v),v\in [n]$. We need to show that w.h.p. we do not come to the end of a list. A given vertex $v$ has probabilty at most $q=\frac{1}{(1-\a-o(1))n}$ of being selected as the next $y$ and this implies that the probability $K=\th\log n$ items on its OUT list are examined is at most $\Pr(Bin(Cn,q)\geq K)=o(n^{-1})$. 

Once $P\cup C=[n]$, it takes $O(n)$ expected time to create a Hamilton cycle. Let us go through the possibilities. 
\begin{enumerate}[(i)]
\item If $C=\La$ and $s_P\in B(f_P)$ or $f_P\in A(s_P)$ then the process finishes in one more step with probability at least $1/n$.
\item If $C=\La$ and (i) does not hold, then we update $out(f_P)$. Note that there is a probability of at least $1-2\a$ that we will now be in case (iii).
\item If $C\neq \La$ then there is a probability of at least $\a$ that $out(f_P)\in C$ and we are in (i).
\end{enumerate}
\subsection{Proof of \eqref{CPstar}}\label{4.1}
We estimate the probability that there is a subset $U, |U|\leq 2\a n$ for which $|\bar{U}^*|\leq \g |U|$ for $\g$ inferred by \eqref{CPstar}. We bound this probability by 
\begin{align}
\sum_{k=1}^{2\a n}\binom{n}{k}\binom{n-k}{\g k}\bfrac{(\g+1)k}{(1-\a)n}^{k\th\log n} &\leq \sum_{k=1}^{2\a n}\bfrac{ne}{k}^k\bfrac{ne}{\g k}^{\g k}\bfrac{(\g+1)k}{(1-\a)n}^{k\th\log n}  \label{exp1}\\
&=\sum_{k=1}^{2\a n}\brac{\bfrac{(\g+1)k}{(1-\a)n}^{\th\log n-\g-1}\cdot (1-\a)e(\g+1)\cdot\bfrac{(\g+1)e}{\g(1-\a)}^\g }^k\nonumber\\
&\leq \sum_{k=1}^{2\a n}\brac{\bfrac{(\g+1)k}{(1-\a)n}^{\th\log n-\g-1}\cdot \g e^{2\g}}^k.\label{exp2}
\end{align}
{\bf Explanation for \eqref{exp1}:} we choose a set of size $k$ for $U$ and then a set of size $\g k$ for $\bar{U}^*$. Then we estimate the probability that each choice in $IN(U)$ is in $U$ or the set of size $\g k$ that we have chosen.

When $k=|U|\geq n/\th\log n$ we take $\g+1=n/10k\leq \th\log n/10$ and we note that $\g  k\geq n/20$. We then have
\[
\sum_{k=n/\th\log n}^{2\a n}\brac{\bfrac{(\g+1)k}{(1-\a)n}^{\th\log n-\g-1}\cdot \g e^{2\g}}^k \leq \sum_{k=n/\th\log n}^{2\a n}\brac{\bfrac{1}{10(1-\a)}^{9\th\log n/10-1}\cdot \th e^{\th\log n/5}\log n}^k=o(1).
\]
For $1\leq k\leq n/\th\log n$ we take $\g=\th^{1/2}\log n$. And then,
\[
\sum_{k=1}^{n/\th\log n}\brac{\bfrac{(\g+1)k}{(1-\a)n}^{\th\log n-\g-1}\cdot \g e^{2\g}}^k \leq \sum_{k=1}^{n/\th\log n}\brac{\bfrac{2}{\th^{1/2}}^{9\th\log n/10}\cdot \th^{1/2}e^{2\th^{1/2}\log n}\log n}^k=o(1).
\]
 This completes the proof of Theorem \ref{th2}

\paragraph{Acknowledgement} We thank Adnane Fouadi for pointing to some defiiciencies in the proof ot Theorem \ref{th1a}.
\section{Conclusion}
We solved the strong connectivity game, but there is a big gap between the upper and lower bounds for Hamiltonicity. Closing this gap is an interesting open problem.

\appendix
\section{Proof of Theorem \ref{th1a}}
We let $G_M,G_B$ denote the subgraphs of $G$ with the edges taken by Maker, Breaker respectively. We let $d_M(v)$ denote the degree of vertex $v$ in $G_M$ for $v\in V$. We define $d_B$ similarly.
Let $\dg(v)=d_B(v)-2bd_M(v)$ be the {\em danger} of vertex $v$ at any time. {\red A vertex is {\em dangerous} for Maker if $d_M(v)<K$.}

{\bf Maker's Strategy:} In round $i$, choose a {\red dangerous} vertex $v_i$ of maximum danger and choose a random edge incident with $v_i$, not already taken. This is called {\em easing} $v$.

We claim that Maker can ensure that for all $v\in V$, we have that {\red $d_B(v)<\a n$ as long as $d_M(v)<K$.} Let $M_i,B_i$ denote Maker and Breaker's $i$th moves. Suppose that Breaker wins in round $g-1$, so that after $B_{g-1}$ there is a vertex $v_g$ such that $d_M(v_g)<K$ and $d_B(v_g)>\a n$. Let $J_i=\set{v_{i+1},\ldots,v_g}$. Next define 
\[
\bdg(M_i)=\frac{\sum_{v\in J_{i-1}}\dg(v)}{|J_{i-1}|}\quad\text{ and }\quad\bdg(B_i)=\frac{\sum_{v\in J_{i}}\dg(v)}{|J_{i}|},
\]
computed before the $i$th moves of Maker, Breaker respectively. 

Then $\bdg(M_1)=0$ and $\bdg(M_g)=\dg(v_g)\geq \a n-2Kb$. Let $a(i)$ be the number of edges contained in $J_i$ that are claimed by \B\ in his first $i$ moves.We have
\begin{lemma}\label{lem1a}
\begin{align}
\bdg(M_i)&\geq \bdg(B_i).\label{appa1}\\
\bdg(M_i)&\geq \bdg(B_i)+\frac{2b}{|J_i|},\text{ if }J_i=J_{i-1}.\label{appa2}\\
\bdg(B_i)&\geq \bdg(M_{i+1})-\frac{2b}{|J_i|}\label{appa3}\\
\bdg(B_i)&\geq \bdg(M_{i+1})-\frac{b+a(i)-a(i-1)}{|J_i|}-1.\label{appa4}
\end{align}
\end{lemma}
\begin{proof}
Equation \eqref{appa1} follows from the fact that a move by \M\ does not increase danger. Equation \eqref{appa2} follows from the fact that if $v_i\in J_{i-1}$ then its danger, which is a maximum, drops by $2b$. Equation \eqref{appa3} follows from the fact that \B\ takes at most $b$ edges inside $J_i$. For equation \ref{appa4}, let $e_{double}$ be the number of edges that \B\ adds to $J_i$ in round $B_i$. Then 
\[
\bdg(B_i)\geq \bdg(M_{i+1})-\frac{b+e_{double}}{|J_i|}
\]
and
\[
a(i)-e_{double}\geq a(i-1)-|J_i|.
\]
\end{proof}
It follows that 
\begin{align}
\bdg(M_i)&\geq \bdg(M_{i+1})\text{ if }J_i=J_{i-1}.\label{appa5}\\
\bdg(M_i)&\geq \bdg(M_{i+1})-\min\set{\frac{2b}{|J_i|},\frac{b+a(i)-a(i-1)}{|J_i|}-1}.\label{appa6}
\end{align}
Next let $1\leq i_1\leq \cdots\leq i_r\leq g-1$ be the indices where $J_i\neq J_{i-1}$. Then we have $|J_{i_r}|=|J_{g-1}|=1$ and $|J_{i_1-1}|=|J_0|=r+1$. Let $k=\frac{n}{\log n}$ and assume first that $r\geq k$ and then use the first minimand in \eqref{appa6} for $i_1,\ldots,i_{r-k}$ and the second minimand otherwise.
\begin{align}
0&=\bdg(M_1)\nonumber\\
&\geq \bdg(M_g)-\frac{b+a(i_r)-a(i_r-1)}{|J_r|}-\cdots- \frac{b+a(i_{r-k+1})-a(i_{r-k+1}-1)}{|J_{r-k+1}|} \nonumber\\
&\hspace{4in}-k-\frac{2b}{|J_{i_{r-k}}|}-\cdots-\frac{2b}{|J_{i_1}|}\label{appa7}\\
&\geq  \bdg(M_g)-\frac{b}{1}-\cdots-\frac{b}{k}-\frac{a(i_r)}{1}-k-\frac{2b}{k+1}-\cdots- \frac{2b}{r}\label{appa8}\\
&\geq \a n-2Kb-b(1+\log k)-k-2b(\log n-\log k).\nonumber
\end{align}
To go from \eqref{appa7} to \eqref{appa8} we use $a(i_{r-j}-1)\geq a(i_{r-j-1}),j>0$ which follows from $J_{i_r-j-1}=J_{i_r-j}-1$ and then the coefficient of $a(i_{r-j-1})$ is at least $\frac{1}{j+1}-\frac{1}{j+2}\geq 0$. Also, $a(i_r)=0$ because $J_{i_r}=J_{g-1}=\set{v_g}$.

It follows that
\[
b\geq \frac{\a n-k}{2K+1+\log n+\log\log n+o(1)}\geq\frac{(\a-1/\log n)n}{(1+2\th+o(1))\log n},
\]
contradicting our upper bound, $2\th<\frac{\a-\b}{\b}$.

If $r<k$ then we replace \eqref{appa8} by
\[
0=\bdg(M_1)\geq \bdg(M_g)-\frac{b}{1}-\cdots-\frac{b}{k}-\frac{a(i_r)}{1}-k\geq \a n-2Kb-b(1+\log k)-k
\]
and obtain the same contradiction.
\end{document}